\documentclass[reqno]{amsart}
\usepackage{amsmath}

\usepackage{amsfonts}
\usepackage{amssymb}
\usepackage{hyperref}
\usepackage{cite}
\begin{document}
\title[ Positive solutions ]{Positive solutions of a nonlinear three-point eigenvalue problem with integral boundary conditions}
\author[F. Haddouchi, S. Benaicha]{Faouzi Haddouchi, Slimane Benaicha}
\address{Faouzi Haddouchi\\
Faculty of Physics, University of Sciences and Technology of
Oran-MB, El Mnaouar, BP 1505, 31000 Oran, Algeria}
\email{fhaddouchi@gmail.com}
\address{Slimane Benaicha \\
Department of Mathematics, University of Oran 1 Ahmed Benbella, 31000 Oran,
Algeria} \email{slimanebenaicha@yahoo.fr}

\begin{abstract}
In this paper, we study the
existence of positive solutions of a three-point integral boundary value
problem (BVP) for the following second-order differential equation
\begin{equation*}
\begin{gathered}
{u^{\prime \prime }}(t)+\lambda a(t)f(u(t))=0,\ \ 0<t<1, \\
u^{\prime}(0)=0, \ u(1)={\alpha}\int_{0}^{\eta}u(s)ds,
\end{gathered}
\end{equation*}
where $\lambda>0$ is a parameter, $0<{\eta}<1$, $0<{\alpha}< \frac{1}{{\eta}}$.
By using the properties of the Green's function and Krasnoselskii's fixed point theorem on cones,
the eigenvalue intervals of the nonlinear boundary value problem are considered, some sufficient conditions for the existence of at least one positive solutions are established.
\end{abstract}

\maketitle \numberwithin{equation}{section}
\newtheorem{theorem}{Theorem}[section]
\newtheorem{lemma}[theorem]{Lemma} \newtheorem{proposition}[theorem]{%
Proposition} \newtheorem{corollary}[theorem]{Corollary} \newtheorem{remark}[%
theorem]{Remark}
\newtheorem{exmp}{Example}[section]

\section{Introduction}
{\footnotesize In this work, we study the existence of positive solutions of a three-point integral boundary value
problem (BVP) for the following second-order differential equation:

\begin{equation} \label{eq-1.1}
{u^{\prime \prime }}(t)+\lambda a(t)f(u(t))=0,\ t\in(0,1),
\end{equation}
\begin{equation} \label{eq-1.2}
u^{\prime}(0)=0, \ u(1)={\alpha}\int_{0}^{\eta}u(s)ds,
\end{equation}
where $0<{\eta}<1$ and $0<{\alpha}<\frac{1}{{\eta}}$, $\lambda$ is a positive parameter, and
\begin{itemize}
\item[(H1)] $f\in C([0,\infty),[0,\infty))$;
\item[(H2)] $a\in C([0,1],[0,\infty))$ and there exists $t_{0}\in[0,\eta]$ such that
$a(t_{0})>0$.
\end{itemize}

The study of the existence of solutions of multi-point boundary value
problems for linear second-order ordinary differential equations was
initiated by II'in and Moiseev \cite{Ilin}.
Then Gupta \cite{Gupt} studied three-point boundary value problems for nonlinear
second-order ordinary differential equations. Since then, the existence of positive
solutions for nonlinear second order three-point boundary-value problems has been
studied by many authors by using the fixed point theorem, nonlinear alternative of the Leray-Schauder
approach, or coincidence degree theory. We refer the reader to
 \cite{Cheng}, \cite{Ander2},\cite{Ander1}, \cite{Feng1}, \cite{Guo},\cite{Han},\cite{He}, \cite{Li},\cite{Luo},\cite{Liang1},\cite{Liang2},\cite{Liu1},\cite{Liu3},\cite{Ma1},\cite{Ma2},
\cite{Ma3},\cite{Ma4},\cite{Ma5}, \cite{Maran},\cite{Pang}, \cite{Sun1},\cite{Sun2},\cite{Tarib},\cite{Webb1},\cite{Xu} and the references therein.

Liu \cite{Liu1} proved the existence of single and multiple positive solutions for the three-point boundary value problem (BVP)
\begin{equation} \label{eq-1.3}
{u^{\prime \prime }}(t)+a(t)f(u(t))=0,\ t\in(0,1),
\end{equation}
\begin{equation} \label{eq-1.4}
u^{\prime}(0)=0, \ u(1)={\beta}u(\eta),
\end{equation}
where $0<{\eta}<1$ and $0<{\beta}<1$.

Recently, Ma \cite{Ma7} studied the second-order three-point boundary value
problem (BVP)
\begin{equation} \label{eq-1.5}
{u^{\prime \prime }}(t)+\lambda a(t)f(u(t))=0,\ t\in(0,1),
\end{equation}
\begin{equation} \label{eq-1.6}
u(0)=\beta u(\eta), \ u(1)={\alpha}u(\eta),
\end{equation}
where $\lambda>0$ is a parameter, $\alpha> 0$, $\beta> 0$, $0 < \eta< 1$, $a\in C([0, 1], [0, \infty))$, $f\in C([0, \infty), [0, \infty))$ and there exists $x_{0}\in(0, 1)$ such that $a(x_{0}) > 0$.
She obtained the existence of single and multiple positive solutions by using Krasnoselskii's fixed point theorem in cones \cite{Krasn}.

Boundary value problems with integral boundary conditions for ordinary differential equations represent a very interesting and important class of problems, and arise in the study of various physical, biological and chemical processes, such as heat conduction, chemical engineering, underground water flow, thermo-elasticity, and plasma physics. They include two, three, multi-point and nonlocal BVPs as special cases. The existence of positive solutions for such class of problems has attracted much attention (see \cite{Bou},\cite{Feng2},\cite{Jiang},\cite{Kang},\cite{Kong},\cite{Purna},\cite{Tarib},\cite{Zhang} and the references therein).

In \cite{Tarib}, Tariboon and Sitthiwirattham investigated the existence of positive solutions of the following three-point integral boundary value problem (BVP)

\begin{equation} \label{eq-1.7}
{u^{\prime \prime }}(t)+a(t)f(u(t))=0,\ t\in(0,1),
\end{equation}
\begin{equation} \label{eq-1.8}
u(0)=0,\ u(1)={\alpha}\int_{0}^{\eta}u(s)ds,
\end{equation}
where $0<{\eta}<1$ and $0<{\alpha}<\frac{2}{{\eta}^{2}}$,
$f\in C([0,\infty),[0,\infty))$, $a\in
C([0,1],[0,\infty))$ and there exists $t_{0}\in[\eta,1]$ such that $a(t_{0})>0$. They showed
the existence of at least one positive solution if $f$ is either superlinear or sublinear by applying Krasnoselskii's fixed point theorem in cones \cite{Krasn}.

In \cite{Haddou1}, by using Leggett-Williams fixed-point theorem, the authors considered the multiplicity of positive solutions of the following three-point integral boundary value problem (BVP)
\begin{equation} \label{eq-1.9}
{u^{\prime \prime }}(t)+f(t, u(t))=0,\ t\in(0,T),
\end{equation}
\begin{equation} \label{eq-1.10}
u(0)={\beta}u(\eta),\ u(T)={\alpha}\int_{0}^{\eta}u(s)ds,
\end{equation}
where $0<{\eta}<T$, $0<{\alpha}<\frac{2T}{{\eta}^{2}}$,
$0\leq{\beta}<\frac{2T-\alpha\eta^{2}}{\alpha\eta^{2}-2\eta+2T}$, and $f\in C([0, T]\times[0,\infty),[0,\infty))$.

Motivated greatly by the above-mentioned excellent works, the aim of this paper is to establish some sufficient conditions for the existence of at least one positive solutions of the BVP \eqref{eq-1.1} and \eqref{eq-1.2}. Our ideas are similar those used in \cite{Ma7}, but a little different.

We firstly give the corresponding Green's function for the associated linear BVP and some of its properties.
Moreover, by applying Krasnoselskii's fixed point theorem, we derive an interval of $\lambda$ on which there exists a positive solution for the three-point integral boundary value problem \eqref{eq-1.1} and \eqref{eq-1.2}.

As applications, some interesting examples are presented to illustrate the main results. The key tool in our approach is the
following Krasnoselskii's fixed point theorem in a cone \cite{Krasn}.

\begin{theorem}\label{theo 1.1}\cite{Krasn}.
Let $E$ be a Banach space, and let $K\subset E$ be a cone.
Assume that $\Omega_{1}$, $\Omega _{2}$ are open bounded subsets of $E$ with $0\in \Omega _{1}$%
, $\overline{\Omega }_{1}\subset \Omega _{2}$, and let

\begin{equation*}\label{eq-11}
A: K\cap (\overline{\Omega }_{2}\backslash  \Omega
_{1})\longrightarrow K
\end{equation*}

be a completely continuous operator such that either

(i) $\ \left\Vert Au\right\Vert \leq \left\Vert u\right\Vert $, $\
u\in K\cap \partial \Omega _{1}$, \ and $\left\Vert Au\right\Vert
\geq \left\Vert u\right\Vert $, $\ u\in K\cap \partial \Omega _{2}$;
or

(ii) $\left\Vert Au\right\Vert \geq \left\Vert u\right\Vert $, $\
u\in K\cap
\partial \Omega _{1}$, \ and $\left\Vert Au\right\Vert \leq \left\Vert
u\right\Vert $, $\ u\in K\cap \partial \Omega _{2}$

hold. Then $A$ has a fixed point in $K\cap
(\overline{\Omega}_{2}\backslash $ $\Omega _{1})$.
\end{theorem}
 }

\section{\protect\footnotesize Preliminaries}

{\footnotesize \medskip
\begin{lemma}
{\footnotesize  \label{lem 2.1} Let $\alpha\eta\neq1$. Then for $y\in C([0,1],\mathbb{R})$, the problem
\begin{equation}\label{eq-2.1}
{u^{\prime \prime }}(t)+y(t)=0, \ t\in (0,1),
\end{equation}
\begin{equation}\label{eq-2.2}
{u^{\prime}}(0)=0, \ u(1)=\alpha \int_{0}^{\eta }u(s)ds,
\end{equation}
has a unique solution
\begin{equation}\label{eq-2.3}
u(t)=\int_{0}^{1}G(t, s)y(s)ds,
\end{equation}
where $G(t, s):[0, 1]\times[0, 1]\rightarrow \mathbb{R}$ is the Green's function defined by
\begin{equation} \label{eq-2.4}
G(t, s)=\frac{1}{2(1-\alpha\eta)}\begin{cases} 2(1-s)-\alpha(\eta-s)^{2}-2(1-\alpha\eta)(t-s),& s\leq \min\{\eta, t\};  \\
2(1-s)-\alpha(\eta-s)^{2}, &t \leq s\leq \eta;\\
2(1-s)-2(1-\alpha\eta)(t-s), &\eta \leq s\leq t; \\
2(1-s), &\max\{\eta, t\}\leq s.
\end{cases}
\end{equation} }
\end{lemma}

\begin{proof}
{\footnotesize From \eqref{eq-2.1}, we have
\begin{equation}\label{eq-2.5}
u(t)=u(0)-\int_{0}^{t}(t-s)y(s)ds.
\end{equation}
Integrating \eqref{eq-2.5} from $0$ to $\eta $, where $\eta \in
(0,1)$, we have
\begin{eqnarray*}
\int_{0}^{\eta }u(s)ds=u(0){\eta}-\frac{1}{2}\int_{0}^{\eta }(\eta -s)^{2}y(s)ds.
\end{eqnarray*}
Since
\begin{equation*}\label{eq-14}
u(1)=u(0)-\int_{0}^{1}(1-s)y(s)ds,
\end{equation*}
from $\ u(1)=\alpha \int_{0}^{\eta }u(s)ds$, we have
\begin{equation*}\label{eq-15}
(1-\alpha \eta )u(0)=\int_{0}^{1}(1-s)y(s)ds-\frac{\alpha }{2}\int_{0}^{\eta }(\eta
-s)^{2}y(s)ds.
\end{equation*}
Therefore,
\begin{eqnarray*}
u(0)=\frac{1}{1-\alpha \eta}\int_{0}^{1}(1-s)y(s)ds -\frac{\alpha} {2(1-\alpha \eta)}\int_{0}^{\eta }(\eta
-s)^{2}y(s)ds,
\end{eqnarray*}
from which it follows that
\begin{eqnarray*}
u(t)&=&\frac{1}{1-\alpha\eta}\int_{0}^{1}(1-s)y(s)ds-\frac{\alpha}{2(1-\alpha\eta)}\int_{0}^{\eta }(\eta
-s)^{2}y(s)ds \\
&&-\int_{0}^{t}(t-s)y(s)ds.
\end{eqnarray*}

If $t\leq\eta$, then
\begin{eqnarray*}
u(t)&=&\int_{0}^{\eta}\frac{1-s}{1-\alpha\eta}y(s)ds-
\int_{0}^{\eta}\frac{\alpha(\eta-s)^{2}}{2(1-\alpha\eta)}y(s)ds+
\int_{\eta}^{1}\frac{1-s}{1-\alpha\eta}y(s)ds\\
&&-\int_{0}^{t}(t-s)y(s)ds \\
&=& \int_{0}^{t}\frac{2(1-s)-\alpha(\eta-s)^{2}-2(1-\alpha\eta)(t-s)}{2(1-\alpha\eta)}y(s)ds+
\int_{\eta}^{1}\frac{1-s}{1-\alpha\eta}y(s)ds\\
&&+\int_{t}^{\eta}\frac{2(1-s)-\alpha(\eta-s)^{2}}{2(1-\alpha\eta)}y(s)ds \\
&=&\int_{0}^{1}G(t, s)y(s)ds.
\end{eqnarray*}
If $t\geq\eta$, then
\begin{eqnarray*}
u(t)&=&\int_{0}^{t}\frac{1-s}{1-\alpha\eta}y(s)ds-
\int_{0}^{\eta}\frac{\alpha(\eta-s)^{2}}{2(1-\alpha\eta)}y(s)ds+
\int_{t}^{1}\frac{1-s}{1-\alpha\eta}y(s)ds\\
&&-\int_{0}^{t}(t-s)y(s)ds \\
&=& \int_{0}^{\eta}\frac{2(1-s)-\alpha(\eta-s)^{2}-2(1-\alpha\eta)(t-s)}{2(1-\alpha\eta)}y(s)ds+
\int_{t}^{1}\frac{1-s}{1-\alpha\eta}y(s)ds\\
&&+\int_{\eta}^{t}\frac{(1-s)-(1-\alpha\eta)(t-s)}{1-\alpha\eta}y(s)ds \\
&=&\int_{0}^{1}G(t, s)y(s)ds.
\end{eqnarray*}
This completes the proof.
 }
\end{proof}
{\footnotesize \medskip For convenience, we define
\[g(s)=\frac{1}{1-\alpha\eta}(1-s), \ \ s\in[0, 1].\]
For the Green's function $G(t, s)$, we have the following two lemmas.}

\begin{lemma}
{\footnotesize  \label{lem 2.2} Let $0<\eta <1$ and $0<\alpha <\frac{1}{\eta}$. Then the
Green's function in \eqref{eq-2.4} satisfies
\begin{equation} \label{eq-2.6}
0\leq G(t, s)\leq g(s),
\end{equation}
for each $s, t\in[0, 1]$.
}
\end{lemma}

\begin{proof}
{\footnotesize First of all, note that by \eqref{eq-2.4} it follows that $G(t, s)\leq g(s)$ for any $(t, s)\in[0, 1]\times[0, 1]$.

Next, we will prove that $G(t, s)\geq0$ for any $(t, s)\in[0, 1]\times[0, 1]$.

If $\eta\leq s\leq t$, then
\begin{eqnarray*}
G(t, s)&=&\frac{1}{1-\alpha\eta}\Big((1-s)-(1-\alpha\eta)(t-s)\Big)\\
&=&\frac{1}{1-\alpha\eta}\Big((1-t)+\alpha\eta(t-s)\Big)\geq0.
\end{eqnarray*}
If $t\leq s\leq \eta$, then
\begin{eqnarray*}
G(t, s)&=&\frac{1}{2(1-\alpha\eta)}\Big(2(1-s)-\alpha(\eta-s)^{2}\Big)\\
&\geq&\frac{1}{2(1-\alpha\eta)}\Big(2(1-s)-\frac{(\eta-s)^{2}}{\eta}\Big)\\
&=&\frac{2\eta -\eta^{2}-s^{2}}{2\eta(1-\alpha\eta)}\\
&\geq&\frac{2\eta -2\eta^{2}}{2\eta(1-\alpha\eta)}\\
&=&\frac{1-\eta}{1-\alpha\eta}\\
&\geq&0.
\end{eqnarray*}
If $s\leq\min\{\eta, t\}$, then
\begin{eqnarray*}
G(t, s)&=&\frac{1}{2(1-\alpha\eta)}\Big(2(1-s)-\alpha(\eta-s)^{2}-2(1-\alpha\eta)(t-s)\Big)\\
&=&\frac{1}{2(1-\alpha\eta)}\Big(2(1-t)-\alpha\eta^{2}-\alpha s^{2}+2\alpha\eta t\Big)\\
&\geq&\frac{1}{2(1-\alpha\eta)}\Big(2(1-t)+\alpha\eta(t-\eta)+\alpha s(\eta-s)\Big).
\end{eqnarray*}
We distinguish the following two cases:

If $s\leq\eta\leq t$, then $G(t, s)\geq0$.

If $s\leq t\leq \eta$, then

\begin{eqnarray*}
G(t, s)&=&\frac{1}{2(1-\alpha\eta)}\Big(2(1-t)-\alpha\eta^{2}-\alpha s^{2}+2\alpha\eta t\Big)\\
&\geq&\frac{1}{2(1-\alpha\eta)}\Big(2(1-t)-\alpha(\eta-t)^{2}\Big)\\
&\geq&\frac{1}{2(1-\alpha\eta)}\Big(2(1-t)-\frac{(\eta-t)^{2}}{\eta}\Big)\\
&=&\frac{2\eta -\eta^{2}-t^{2}}{2\eta(1-\alpha\eta)}\\
&\geq&\frac{2\eta-2\eta^{2}}{2\eta(1-\alpha\eta)}\\
&=&\frac{1-\eta}{1-\alpha\eta}\\
&\geq&0.
\end{eqnarray*}
The proof is completed.}
\end{proof}

\begin{lemma}
{\footnotesize  \label{lem 2.3}
Let $0<\eta <1$ and $0<\alpha <\frac{1}{\eta}$. Then for any $(t, s)\in[0, \eta]\times[0, 1]$, $G(t, s)\geq\gamma g(s)$, where
\begin{equation} \label{eq-2.7}
0<\gamma=1-\eta<1.
\end{equation}}
\end{lemma}
\begin{proof}
{\footnotesize If $s=1$, then by Lemma \ref{lem 2.2}, the result follows. Now we suppose that $(t, s)\in[0, \eta]\times[0, 1)$.

If $t\leq s\leq \eta$, then
\begin{eqnarray*}
\frac{G(t, s)}{g(s)}&=&\frac{2(1-s)-\alpha(\eta-s)^{2}}{2(1-s)}\\
&\geq&\frac{2(1-s)-\frac{(\eta-s)^{2}}{\eta}}{2(1-s)}\\
&=&\frac{2\eta -\eta^{2}-s^{2}}{2\eta(1-s)}\\
&\geq&\frac{1-\eta}{1-s}\\
&\geq& \gamma.
\end{eqnarray*}

If $s\leq t\leq \eta$, then
\begin{eqnarray*}
\frac{G(t, s)}{g(s)}&=&\frac{2(1-s)-\alpha(\eta-s)^{2}-2(1-\alpha\eta)(t-s)}{2(1-s)}\\
&=&\frac{2(1-t)-\alpha\eta^{2}-\alpha s^{2}+2\alpha\eta t}{2(1-s)}\\
&\geq&\frac{2(1-t)-\alpha(\eta-t)^{2}}{2\eta(1-s)}\\
&\geq&\frac{1-\eta}{1-s}\\
&\geq& \gamma.
\end{eqnarray*}

If $t\leq \eta\leq s$, then

\begin{eqnarray*}
\frac{G(t, s)}{g(s)}&=&1\geq \gamma.
\end{eqnarray*}
Therefore,

\[G(t, s)\geq\gamma g(s),\ \ (t, s)\in[0, \eta]\times[0, 1].\]

The proof is completed. }
\end{proof}

{\footnotesize \medskip Let $E = C([0,1],\mathbb{R})$, and only the sup norm is used. It is easy to see that the BVP \eqref{eq-1.1} and \eqref{eq-1.2} has a solution  $u = u(t)$ if and only if  $u$ is a fixed point of operator $A_{\lambda}$, where
$A_{\lambda}$ is defined by
\begin{eqnarray*}
A_{\lambda}u(t)&=&\frac{\lambda}{1-\alpha\eta}\int_{0}^{1}(1-s)a(s)f(u(s))ds-
\frac{\lambda\alpha}{2(1-\alpha\eta)}\int_{0}^{\eta }(\eta-s)^{2}a(s)f(u(s))ds \\
&&-\lambda\int_{0}^{t}(t-s)a(s)f(u(s))ds\\
&=&\lambda \int_{0}^{1}G(t, s)a(s)f(u(s))ds.
\end{eqnarray*}

Denote
\begin{equation}\label{eq-2.8}
K=\left\{u\in E: u\geq0,  \min_{t\in
[0,\eta]}u(t)\geq \gamma \|u\|\right\},
\end{equation}
where $\gamma$ is defined in \eqref{eq-2.7}.
It is obvious that $K$ is a cone in $E$.}

\begin{lemma}
{\footnotesize  \label{lem 2.4} Assume that {\rm (H1)} and {\rm (H2)} hold, $0<\eta <1$ and $0<\alpha <\frac{1}{\eta}$. Then the operator $A_{\lambda}:K\rightarrow K$ is completely continuous.}
\end{lemma}

\begin{proof}
{\footnotesize
For $u\in K$, according to the definition of $A_{\lambda}$, Lemma \ref{lem 2.2} and Lemma \ref{lem 2.3}, it is easy to prove that $A_{\lambda}K\subset K$. By the Ascoli-Arzela theorem,
it is easy to show that $A_{\lambda}:K\rightarrow K$ is completely continuous.}
\end{proof}

{\footnotesize \medskip In what follows, for the sake of convenience, set\\
\begin{equation*}\label{eq-20}
\Lambda_{1}=\frac{1}{1-\alpha\eta}\int_{0}^{1}(1-s)a(s)ds, \ \  \Lambda_{2}=\frac{\gamma}{2(1-\alpha\eta)}\int_{0}^{\eta}\Big(2(1-\eta)+\alpha(\eta^{2}-s^{2})\Big)a(s)ds,
\end{equation*}
\begin{equation*}\label{eq-21}
f_0=\lim_{u\to 0+}\frac{f(u)}{u}, \ \
f_{\infty}=\lim_{u\to\infty}\frac{f(u)}{u}.
\end{equation*}
}}
\section{\protect\footnotesize Main results}

{\footnotesize \medskip In this section, we will state and prove our main results.

\begin{theorem}
{\footnotesize \label{theo 3.1}
Suppose that {\rm (H1)} and {\rm (H2)} hold, $0<\eta <1$ and $0<\alpha <\frac{1}{\eta}$. If $\Lambda_{1}f_{0}<\Lambda_{2}f_{\infty}$, then for each $\lambda\in(\frac{1}{\Lambda_{2}f_{\infty}},\frac{1}{\Lambda_{1}f_{0}})$, the BVP \eqref{eq-1.1} and \eqref{eq-1.2}
has at least one positive solution.}
\end{theorem}

\begin{proof}
{\footnotesize Let $\lambda\in(\frac{1}{\Lambda_{2}f_{\infty}}, \frac{1}{\Lambda_{1}f_{0}})$, and choose $\epsilon>0$ such that
\begin{equation}\label{eq-3.1}
\frac{1}{\Lambda_{2}(f_{\infty}-\epsilon)}\leq\lambda\leq\frac{1}{\Lambda_{1}(f_{0}+\epsilon)}.
\end{equation}

By the definition of $f_{0}$, there exists $\rho_{1}>0$ such that

\begin{equation}\label{eq-3.2}
f(u)\leq(f_{0}+\epsilon)u,\ \text{for}\ u\in(0,\rho_{1}].
\end{equation}

Let $\Omega_{\rho_{1}}=\left\{u\in E: \|u\|<\rho_{1}\right\}$, then from \eqref{eq-3.1}, \eqref{eq-3.2} and Lemma \ref{lem 2.2}, for any $u\in K\cap \partial\Omega_{\rho_{1}}$, we have
\begin{eqnarray*}
A_{\lambda}u(t)&=&\lambda \int_{0}^{1}G(t, s)a(s)f(u(s))ds \\
&\leq&\lambda \int_{0}^{1}g(s)a(s)f(u(s))ds \\
&=&\frac{\lambda}{1-\alpha\eta}\int_{0}^{1}(1-s)a(s)f(u(s))ds\\
&\leq&\frac{\lambda}{1-\alpha\eta}\int_{0}^{1}(1-s)a(s)(f_{0}+\epsilon)u(s)ds\\
&\leq&\lambda\Lambda_{1}(f_{0}+\epsilon)\|u\|\leq\|u\|,
\end{eqnarray*}
which yields
\begin{equation}\label{eq-3.3}
\|A_{\lambda}u\|\leq\|u\|, \ \ \text{for}\ u\in K\cap\partial\Omega_{\rho_{1}}.
\end{equation}
Further, by the definition of $f_{\infty}$, there exists $\widehat{\rho}_{2}>0$ such that

\begin{equation}\label{eq-3.4}
f(u)\geq(f_{\infty}-\epsilon)u,\ \text{for}\ u\in[\widehat{\rho}_{2}, \infty).
\end{equation}

Now, set $\rho_{2}=\max\left\{2\rho_{1}, \frac{\widehat{\rho}_{2}}{\gamma}\right\}$ and $\Omega_{\rho_{2}}=\left\{u\in E: \|u\|<\rho_{2}\right\}$. Then $u\in K\cap
\partial\Omega_{\rho_{2}}$ implies that
\[ u(t)\geq\gamma\|u\|\geq\widehat{\rho}_{2}, \ \ t\in[0, \eta],\]
and so,

\begin{eqnarray*}
A_{\lambda}u(\eta)&=&\lambda \int_{0}^{1}G(\eta, s)a(s)f(u(s))ds \\
&\geq&\lambda \int_{0}^{\eta}G(\eta, s)a(s)f(u(s))ds \\
&\geq&\lambda \int_{0}^{\eta}G(\eta, s)a(s)(f_{\infty}-\epsilon)u(s)ds \\
&\geq&\lambda \gamma(f_{\infty}-\epsilon)\|u\|\int_{0}^{\eta}G(\eta, s)a(s)ds\\
&=&\lambda\gamma(f_{\infty}-\epsilon)\|u\|\int_{0}^{\eta}
\frac{2(1-s)-\alpha(\eta-s)^{2}-2(1-\alpha\eta)(\eta-s)}{2(1-\alpha\eta)}a(s)ds\\
&=&\lambda(f_{\infty}-\epsilon)\|u\|\frac{\gamma}{2(1-\alpha\eta)}\int_{0}^{\eta}\Big(2(1-\eta)+\alpha(\eta^{2}-s^{2})\Big)a(s)ds\\
&=&\lambda\Lambda_{2}(f_{\infty}-\epsilon)\|u\|\geq\|u\|.
\end{eqnarray*}
This implies that

\begin{equation}\label{eq-3.5}
\|A_{\lambda}u\|\geq\|u\|, \ \ \text{for}\ u\in K\cap\partial\Omega_{\rho_{2}}.
\end{equation}
Therefore, from \eqref{eq-3.3}, \eqref{eq-3.5} and Theorem \ref{theo 1.1}, it follows that $A_{\lambda}$ has a fixed point $u$ with $\rho_{1}\leq\|u\|\leq\rho_{2}$ in $K\cap (\overline{\Omega}_{\rho_{2}}\backslash \Omega _{\rho_{1}})$, which is a desired positive solution of the BVP \eqref{eq-1.1} and \eqref{eq-1.2}.}
\end{proof}

{\footnotesize \medskip By Theorem \ref{theo 3.1} we can easily obtain the following corollary.}

\begin{corollary}
{\footnotesize  \label{cor 3.1}
Assume that {\rm (H1)} and {\rm (H2)} hold, $0<\eta <1$ and $0<\alpha <\frac{1}{\eta}$. Then we have
\begin{enumerate}
  \item If $f_{0}=0,\ f_{\infty}=\infty$, then for each $\lambda\in(0, \infty)$, the BVP \eqref{eq-1.1} and \eqref{eq-1.2} has at least one positive solution.
  \item If $f_{\infty}=\infty, \ 0<f_{0}<\infty$, then for each $\lambda\in(0, \frac{1}{\Lambda_{1}f_{0}})$, the BVP \eqref{eq-1.1} and \eqref{eq-1.2} has at least one positive solution.
  \item If $f_{0}=0, \ 0<f_{\infty}<\infty$, then for each $\lambda\in(\frac{1}{\Lambda_{2}f_{\infty}}, \infty)$, the BVP \eqref{eq-1.1} and \eqref{eq-1.2} has at least one positive solution.
\end{enumerate}}
\end{corollary}

\begin{theorem}
{\footnotesize \label{theo 3.2}
Suppose that {\rm (H1)} and {\rm (H2)} hold, $0<\eta <1$ and $0<\alpha <\frac{1}{\eta}$. If $\Lambda_{1}f_{\infty}<\Lambda_{2}f_{0}$, then for each $\lambda\in(\frac{1}{\Lambda_{2}f_{0}},\frac{1}{\Lambda_{1}f_{\infty}})$, the BVP \eqref{eq-1.1} and \eqref{eq-1.2} has at least one positive solution.}
\end{theorem}

\begin{proof}
{\footnotesize
Let $\lambda\in(\frac{1}{\Lambda_{2}f_{0}}, \frac{1}{\Lambda_{1}f_{\infty}})$, and choose $\epsilon>0$ such that
\begin{equation}\label{eq-3.6}
\frac{1}{\Lambda_{2}(f_{0}-\epsilon)}\leq\lambda\leq\frac{1}{\Lambda_{1}(f_{\infty}+\epsilon)}.
\end{equation}

By the definition of $f_{0}$, there exists $\rho_{1}>0$ such that

\begin{equation}\label{eq-3.7}
f(u)\geq(f_{0}-\epsilon)u,\ \text{for}\ u\in(0,\rho_{1}].
\end{equation}

Let $\Omega_{\rho_{1}}=\left\{u\in E: \|u\|<\rho_{1}\right\}$. Hence, for any $u\in K\cap \partial\Omega_{\rho_{1}}$, from \eqref{eq-3.6}, \eqref{eq-3.7}, we get
\begin{eqnarray*}
A_{\lambda}u(\eta)&\geq&\lambda \int_{0}^{\eta}G(\eta, s)a(s)f(u(s))ds \\
&\geq&\lambda \int_{0}^{\eta}G(\eta, s)a(s)(f_{0}-\epsilon)u(s)ds \\
&\geq&\lambda\Lambda_{2}(f_{0}-\epsilon)\|u\|\geq\|u\|.
\end{eqnarray*}
Therefore
\begin{equation}\label{eq-3.8}
\|A_{\lambda}u\|\geq\|u\|, \ \ \text{for}\ u\in K\cap\partial\Omega_{\rho_{1}}.
\end{equation}

By the definition of $f_{\infty}$, there exists $\rho_{0}>0$ such that

\begin{equation}\label{eq-3.9}
f(u)\leq(f_{\infty}+\epsilon)u,\ \text{for}\ u\in[\rho_{0}, \infty).
\end{equation}
Next, we consider two cases:

If $f$ is bounded. Let $f(u)\leq L$ for all $u\in [0,\infty)$. Set
$\rho_{2}=\max\left\{2\rho_{1}, \lambda \Lambda _{1} L\right\}$ and $\Omega_{\rho_{2}}=\left\{u\in E: \|u\|<\rho_{2}\right\}$, then for $u\in K\cap \partial\Omega_{\rho_{2}}$, we have
\begin{eqnarray*}
A_{\lambda}u(t)&\leq&\frac{\lambda}{1-\alpha\eta}\int_{0}^{1}(1-s)a(s)f(u(s))ds\\
&\leq&\lambda L\Lambda_{1}\leq\rho_{2}\leq\|u\|.
\end{eqnarray*}
Therefore
\begin{equation}\label{eq-3.10}
\|A_{\lambda}u\|\leq\|u\|, \ \ \text{for}\ u\in K\cap\partial\Omega_{\rho_{2}}.
\end{equation}
If $f$ is unbounded, then from $f\in C([0,\infty),[0,\infty))$, we know that there is $\rho_{2}$: 
$\rho_{2}\geq\max\left\{2\rho_{1}, \gamma^{-1}\rho_{0}\right\}$ such  that
\begin{equation}\label{eq-3.11}
f(u)\leq f(\rho_{2}),\ \text{for}\ u\in\left[0,\rho_{2}\right].
\end{equation}
Let $\Omega_{\rho_{2}}=\left\{u\in E: \|u\|<\rho_{2}\right\}$, then for $u\in K\cap \partial\Omega_{\rho_{2}}$, we have
\begin{eqnarray*}
A_{\lambda}u(t)&\leq&\frac{\lambda}{1-\alpha\eta}\int_{0}^{1}(1-s)a(s)f(u(s))ds\\
&\leq&\frac{\lambda}{1-\alpha\eta}\int_{0}^{1}(1-s)a(s)f(\rho_{2})ds\\
&\leq&\lambda\Lambda_{1}\rho_{2}(f_{\infty}+\epsilon)\leq\rho_{2}=\|u\|.
\end{eqnarray*}
Thus
\begin{equation}\label{eq-3.12}
\|A_{\lambda}u\|\leq\|u\|, \ \ \text{for}\ u\in K\cap\partial\Omega_{\rho_{2}}.
\end{equation}
It follows from Theorem \ref{theo 1.1}, that $A_{\lambda}$ has a fixed point in $K\cap (\overline{\Omega}_{\rho_{2}}\backslash \Omega _{\rho_{1}})$, such that
$\rho_{1}\leq\|u\|\leq\rho_{2}$.}
\end{proof}
{\footnotesize \medskip From Theorem \ref{theo 3.2}, we have}

\begin{corollary}
{\footnotesize  \label{cor 3.2}
Assume that {\rm (H1)} and {\rm (H2)} hold, $0<\eta <1$ and $0<\alpha <\frac{1}{\eta}$. Then we have
\begin{enumerate}
  \item If $f_{0}=\infty,\ f_{\infty}=0$, then for each $\lambda\in(0, \infty)$, the BVP \eqref{eq-1.1} and \eqref{eq-1.2} has at least one positive solution.
  \item If $f_{\infty}=0, \ 0<f_{0}<\infty$, then for each $\lambda\in(\frac{1}{\Lambda_{2}f_{0}}, \infty)$, the BVP \eqref{eq-1.1} and \eqref{eq-1.2} has at least one positive solution.
  \item If $f_{0}=\infty, \ 0<f_{\infty}<\infty$, then for each $\lambda\in(0, \frac{1}{\Lambda_{1}f_{\infty}})$, the BVP \eqref{eq-1.1} and \eqref{eq-1.2} has at least one positive solution.
\end{enumerate}}
\end{corollary}
}
\section{\protect\footnotesize Examples}
{\footnotesize \medskip In this section we present some examples to illustrate our main results.
\begin{exmp}
{\rm \label{eq-1} Consider the boundary value problem
\begin{equation}\label{eq-4.1}
{u^{\prime \prime }}(t)+tu^{p}=0, \  \ 0<t<1,
\end{equation}
\begin{equation}\label{eq-4.2}
{u^{\prime}}(0)=0, \  \ u(1)= 2\int_{0}^{\frac{1}{4}}u(s)ds.
\end{equation}
Set $\alpha=2$, $\eta=1/4$, $a(t)=t$, $f(u)=u^{p}$ $(p\in (0, 1)\cup(1, \infty)$. We can show that
$0<\alpha=2<4=1/\eta$.\\
Now we consider the existence of positive solutions of the problem \eqref{eq-4.1} and \eqref{eq-4.2} in two
cases.\\
Case 1: $p>1$. In this case, $f_{0}=0$, $f_{\infty}=\infty$. Then, by Corollary \ref{cor
3.1}, the BVP \eqref{eq-4.1} and \eqref{eq-4.2} has at least one positive solution.\\
Case 2: $p\in (0, 1)$. In this case, $f_{0}=\infty$, $f_{\infty}=0$. Then, by Corollary \ref{cor
3.2}, the BVP \eqref{eq-4.1} and \eqref{eq-4.2} has at least one positive solution. }
\end{exmp}

\begin{exmp}
{\rm \label{eq-2} Consider the boundary value problem
\begin{equation}\label{eq-4.3}
{u^{\prime \prime }}(t)+\frac{aue^{2u}}{b+e^{u}+e^{2u}}=0, \  \ 0<t<1,
\end{equation}
\begin{equation}\label{eq-4.4}
{u^{\prime}}(0)=0, \  \ u(1)=2\int_{0}^{\frac{1}{3}}u(s)ds.
\end{equation}
Set $\alpha=2$, $\eta=1/3$, $a(t)\equiv1$, $f(u)=(aue^{2u})/(b+e^{u}+e^{2u})$.
we consider the existence of positive solutions of the problem \eqref{eq-4.3} and \eqref{eq-4.4} into the following two cases.

Case 1: If $a=5$, $b=8$, then $f_{0}=\frac{1}{2}$, $f_{\infty}=5$. By calculating, it is easy to obtain that $0<\alpha=2<3=1/\eta$, $\gamma=\frac{2}{3}$. Again
\begin{eqnarray*}
\Lambda_{1}&=&\frac{1}{1-\alpha\eta}\int_{0}^{1}(1-s)a(s)ds=\frac{3}{2},\\
\Lambda_{2}&=&\frac{\gamma}{2(1-\alpha\eta)}\int_{0}^{\eta}\Big(2(1-\eta)+\alpha(\eta^{2}-s^{2})\Big)a(s)ds=\frac{40}{81},\\
\Lambda_{1}f_{0}&=&\frac{3}{4}, \ \ \Lambda_{2}f_{\infty}=\frac{200}{81}.
\end{eqnarray*}

By Theorem \ref{theo 3.1}, we know that for any $\lambda\in(\frac{81}{200}, \frac{4}{3})$, the BVP \eqref{eq-4.3} and \eqref{eq-4.4} has at least one positive solution $u\in C[0, 1]$.

Case 2: If $a=5$, $b=-2$, then $f_{0}=\infty$, $f_{\infty}=5$ and $\Lambda_{1}f_{\infty}=\frac{15}{2}$.

Therefore, by Corollary \ref{cor 3.2}, we know that for any $\lambda\in(0, \frac{2}{15})$, the BVP \eqref{eq-4.3} and \eqref{eq-4.4} has at least one positive solution $u\in C[0, 1]$.}
\end{exmp}

\begin{exmp}
{\rm \label{eq-3}
Consider the boundary value problem
\begin{equation}\label{eq-4.5}
{u^{\prime \prime }}(t)+\frac{1}{5}u(1-\frac{1}{1+u^{2}})=0, \  \ 0<t<1,
\end{equation}

\begin{equation}\label{eq-4.6}
{u^{\prime}}(0)=0, \  \ u(1)=\int_{0}^{\frac{1}{2}}u(s)ds,
\end{equation}
where $\alpha=1$, $\eta=1/2$, $a(t)\equiv\frac{1}{5}$,
$f(u)=u(1-\frac{1}{1+u^{2}})$. By calculating, we have $0<\alpha=1<2=1/\eta$, $\gamma=\frac{1}{2}$, $f_{0}=0$, $f_{\infty}=1$. Again

\begin{eqnarray*}
\Lambda_{2}&=&\frac{\gamma}{2(1-\alpha\eta)}\int_{0}^{\eta}\Big(2(1-\eta)+\alpha(\eta^{2}
-s^{2})\Big)a(s)ds=\frac{7}{120}.\\
\end{eqnarray*}
Hence, by Corollary \ref{cor
3.1}, for any $\lambda\in(\frac{120}{7},\infty)$, the BVP \eqref{eq-4.5} and \eqref{eq-4.6} has at least one positive solution $u\in C[0, 1]$.}
\end{exmp}}


\begin{thebibliography}{9}

\bibitem{Ander1}
D. R. Anderson, Nonlinear triple-point problems on time scales, \textit{Electron. J. Diff. Eqns.} \textbf{47} (2004), 1--12.

\bibitem{Ander2}
\bysame, Solutions to second order three-point problems on time
  scales, \textit{J. Difference Equ. Appl.} \textbf{8} (2002), 673--688.

\bibitem{Bou}
A. Boucherif, Second-order boundary value problems with integral boundary conditions, \textit{Nonlinear Anal.} \textbf{70} (2009), 364--371.

\bibitem{Cheng}
Z. Chengbo, Positive solutions for semi-positone three-point boundary
  value problems, \textit{J. Comput. Appl. Math.} \textbf{228} (2009), 279--286.

\bibitem{Feng1}
W. Feng and J. R. L. Webb, Solvability of a three-point nonlinear boundary value
  problem at resonance, \textit{Nonlinear Anal.} \textbf{30} (1997), no. 6, 3227--3238.

\bibitem{Feng2}
M. Feng, D. Ji, and W. Ge, Positive solutions for a class of boundary-value problem with integral boundary conditions in Banach spaces, \textit{J. Comput. Appl. Math.} \textbf{222} (2008), 351-–363.


\bibitem{Guo}
Y. Guo and W. Ge, Positive solutions for three-point boundary value problems
  with dependence on the first order derivative, \textit{J. Math. Anal. Appl.} \textbf{290} (2004), 291--301.


\bibitem{Gupt}
C. P. Gupta, Solvability of a three-point nonlinear boundary value problem
  for a second order ordinary differential equations, \textit{J. Math. Anal.
  Appl.} \textbf{168} (1992), 540--551.

\bibitem{Haddou1}
F. Haddouchi and S. Benaicha, Multiple positive solutions for a nonlinear
three-point integral boundary-value problem, \textit{Int. J. Open Probl. Comput. Sci. Math.} \textbf{8} (2015), no. 1, 29-42.

\bibitem{Han}
X. Han, Positive solutions for a three-point boundary value problem, \textit{Nonlinear Anal.} \textbf{66} (2007), 679--688.

\bibitem{He}
X. He and W. Ge, Triple solutions for second order three-point boundary
  value problems, \textit{J. Math. Anal. Appl.} \textbf{268} (2002), 256--265.


\bibitem{Ilin}
 V. A. Il'in and E. I. Moiseev, Nonlocal boundary-value problem of the first kind
  for a Sturm-Liouville operator in its differential and finite difference
  aspects, \textit{Differ. Equ.} \textbf{23} (1987), no. 7, 803--810.

\bibitem{Jiang}
J. Jiang, L. Liu, and Y. Wu, Second-order nonlinear singular Sturm-Liouville problems with integral boundary conditions, \textit{Appl. Math. Comput.} \textbf{215} (2009), 1573–-1582.



\bibitem{Kang}
P. Kang and Z. Wei, Three positive solutions of singular nonlocal boundary value problems for systems of nonlinear second-order ordinary differential equations, \textit{Nonlinear Anal.} \textbf{70} (2008), 444-–451.


\bibitem{Kong}
L. Kong, Second order singular boundary value problems with integral boundary conditions, \textit{Nonlinear Anal.} \textbf{72} (2010), 2628–-2638.

\bibitem{Krasn}
M. A. Krasnoselskii, \textit{Positive Solutions of Operator Equations}, P. Noordhoff, Groningen, The Netherlands, 1964.


\bibitem{Li}
J. Li and J. Shen, Multiple positive solutions for a second-order
  three-point boundary value problem, \textit{Appl. Math. Comput.} \textbf{182} (2006), 258--268.


\bibitem{Liang1}
S. Liang and L. Mu, Multiplicity of positive solutions for singular
  three-point boundary value problems at resonance, \textit{Nonlinear Anal.} \textbf{71} (2009), 2497--2505.

\bibitem{Liang2}
R. Liang, J. Peng, and J. Shen, Positive solutions to a generalized second
  order three-point boundary value problem, \textit{Appl. Math. Comput.} \textbf{196} (2008), 931--940.


\bibitem{Liu1}
B. Liu, Positive solutions of a nonlinear three-point boundary value
  problem, \textit{Appl. Math. Comput.} \textbf{132} (2002), 11--28.


\bibitem{Liu3}
\bysame, Positive solutions of a nonlinear three-point boundary value
  problem, \textit{Comput. Math. Appl.} \textbf{44} (2002), 201--211.

\bibitem{Luo}
H. Luo and Q. Ma, Positive solutions to a generalized second-order
  three-point boundary-value problem on time scales, \textit{Electron. J. Diff.
  Eqns.} \textbf{17} (2005), 1--14.


\bibitem{Ma1}
R. Ma, Existence theorems for a second order three-point boundary value
  problem, \textit{J. Math. Anal. Appl.} \textbf{212} (1997), 430--442.

\bibitem{Ma2}
\bysame, Multiplicity of positive solutions for second-order three-point
  boundary value problems, \textit{Comput. Math. Appl.} \textbf{40} (2000), 193--204.

 \bibitem{Ma3}
\bysame, Positive solutions for a nonlinear three-point boundary value
  problem, \textit{Electron. J. Diff. Eqns.} \textbf{34} (1999), 1--8.

\bibitem{Ma4}
\bysame, Positive solutions for second-order three-point boundary value
  problems, \textit{Appl. Math. Lett.} \textbf{14} (2001), 1--5.


\bibitem{Ma5}
R. Ma and H. Wang, Positive solutions of nonlinear three-point boundary
  value problem, \textit{J. Math. Anal. Appl.} \textbf{279} (2003), 216--227.


\bibitem{Ma7}
Q. Ma, Existence of positive solutions for the symmetry three-point boundary
  value problem, \textit{Electron. J. Diff. Eqns.} \textbf{154} (2007), 1--8.


\bibitem{Maran}
S. A. Marano, A remark on a second order three-point boundary value
  problem, \textit{J. Math. Anal. Appl.} \textbf{183} (1994), 581--522.


\bibitem{Pang}
H. Pang, M. Feng, and W. Ge, Existence and monotone iteration of positive
  solutions for a three-point boundary value problem, \textit{Appl. Math. Lett.} \textbf{21} (2008), 656--661.


\bibitem{Purna}
I. K. Purnaras, On the existence of nonnegative solutions to an integral equation with applications to boundary value problems, \textit{Nonlinear Anal.} \textbf{71} (2009), 3914--3933.

\bibitem{Sun1}
H. R. Sun and W. T. Li, Positive solutions for nonlinear three-point boundary
  value problems on time scales, \textit{J. Math. Anal. Appl.} \textbf{299} (2004), 508--524.

\bibitem{Sun2}
Y. Sun, L. Liu, J. Zhang, and R. P. Agarwal, Positive solutions of singular
  three-point boundary value problems for second-order differential equations, \textit{J. Comput. Appl. Math.} \textbf{230} (2009), 738--750.


\bibitem{Tarib}
J. Tariboon and T. Sitthiwirattham, Positive solutions of a nonlinear
  three-point integral boundary value problem, \textit{Bound. Val. Prob.} ID 519210, doi:10.1155/2010/519210
  (2010), 11 pages.


\bibitem{Webb1} J. R. L. Webb, Positive solutions of some three-point boundary value problems via fixed point index theory, \textit{Nonlinear Anal.} \textbf{47} (2001),  no. 7, 4319--4332.

\bibitem{Xu}
X. Xu, Multiplicity results for positive solutions of some semi-positone
  three-point boundary value problems, \textit{J. Math. Anal. Appl.} \textbf{291} (2004), 673--689.

\bibitem{Zhang}
X. Zhang, M. Feng, and W. Ge, Symmetric positive solutions for p-Laplacian fourth-order differential equations with integral boundary conditions, \textit{J. Comput. Appl. Math.} \textbf{222} (2008), 561–-573.

\end{thebibliography}
\end{document}